\documentclass[11pt]{article}

\usepackage{amsmath, amsfonts,amssymb, amsthm}
\usepackage{mathtools}
\usepackage{multirow}
\usepackage{graphicx}
\usepackage{color}
\usepackage[table]{xcolor}
\usepackage{epsf}
\usepackage{tikz}
\usepackage{subfig}
\usepackage{hyperref}
\hypersetup{
    colorlinks = true,
    linkcolor = {black},
    citecolor = {black}
}

\usepackage[left=1.3in, right=1.3in, top=1.2in, bottom=1.25in]{geometry}

\usepackage{enumitem}
\setlist{itemsep=0pt, topsep=0pt}

\newcommand{\floor}[1]{\lfloor#1\rfloor}
\newcommand{\ceiling}[1]{\lceil#1\rceil}
\newcommand{\bigfloor}[1]{\left\lfloor#1\right\rfloor}

\newcommand{\set}[1]{\left\{ #1 \right\}}

\newcommand{\cG}{\mathcal{G}}

\newcommand{\cS}{\mathcal{S}}

\newcommand{\tbf}[1]{\textbf{#1}}

\newtheorem{theorem}{Theorem}[section]
\newtheorem{corollary}[theorem]{Corollary}
\newtheorem{lemma}[theorem]{Lemma}

\newtheorem{claim}[theorem]{Claim}

\newtheorem{observation}[theorem]{Observation}

\newtheorem{example}[theorem]{Example}
\newtheorem{problem}[theorem]{Problem}

\newenvironment{proofclaim}[1][Proof of claim]{\begin{proof}[#1]}{\end{proof}}

\newcommand{\ep}{\epsilon}

\newcommand{\bip}{\mathrm{bip}}

\newcommand{\of}[1]{\left( #1 \right)}

\usepackage{color}

\usepackage{xcolor}

\usepackage{caption}
\title{On the multicolor Ramsey numbers of balanced double stars}
\author{Deepak Bal\thanks{Department of Mathematics, Montclair State University, Montclair, NJ {\tt deepak.bal@montclair.edu}.}, Louis DeBiasio\thanks{Department of Mathematics, Miami University, Oxford, OH. \texttt{debiasld@miamioh.edu}. Research supported in part by NSF grant DMS-1954170.}, Ella Oren-Dahan\thanks{Department of Mathematics, Montclair State University, Montclair, NJ. {\tt orendahane1@montclair.edu}}}
\date{\today}

\begin{document}

\maketitle

\begin{abstract}
The balanced double star on $2n+2$ vertices, denoted $S_{n,n}$, is the tree obtained by joining the centers of two disjoint stars each having $n$ leaves.  Let $R_r(G)$ be the smallest integer $N$ such that in every $r$-coloring of the edges of $K_N$ there is a monochromatic copy of $G$, and let $R_r^{\bip}(G)$ be the smallest integer $N$ such that in every $r$-coloring of the edges of $K_{N,N}$ there is a monochromatic copy of $G$.  It is known that $R_2(S_{n,n})=3n+2$ \cite{GHK} and $R_2^{\bip}(S_{n,n})=2n+1$ \cite{HJ}, but very little is known about $R_r(S_{n,n})$ and $R^{\mathrm{bip}}_r(S_{n,n})$ when $r\geq 3$ (other than the bounds which follow from considerations on the number of edges in the majority color class).

In this paper we prove the following for all $n\geq 1$ (where the lower bounds are adapted from existing examples):
\begin{itemize}
\item $(r-1)2n+1\leq R_r(S_{n,n})\leq (r-\frac{1}{2})(2n+2)-1$, and
\item $(2r-4)n+1\leq R^{\mathrm{bip}}_r(S_{n,n})\leq (2r-3+\frac{2}{r}+O(\frac{1}{r^2}))n.$
\end{itemize}
These bounds are similar to the best known bounds on $R_r(P_{2n+2})$ and $R_r^{\mathrm{bip}}(P_{2n+2})$, where $P_{2n+2}$ is a path on $2n+2$ vertices (which is also a balanced tree).

We also give an example which improves the lower bound on $R^{\mathrm{bip}}_r(S_{n,n})$ when $r=3$ and $r=5$.  
\end{abstract}

\section{Introduction}

A double star is a tree obtained by joining the centers of two disjoint stars.  Let $S_{k,l}$ be the double star obtained by joining the centers of two disjoint stars having $k$ and $l$ leaves respectively.  So $|V(S_{k,l})|=k+l+2$.  Let $\cS_n$ be the family of all double stars on $n$ vertices. 

Given a graph $K$, a family of graphs $\cG$, and a positive integer $r$, we write $K\to_{r} \cG$ to mean that in every $r$-coloring of the edges of $K$ there is a monochromatic copy of some $G\in \cG$.  We let $R_r(\cG)$ be the smallest integer $N$ such that $K_N\to_r \cG$.  When $\cG$ consists of a single graph $G$, we simply write $K\to_r G$ and $R_r(G)$.  As is customary, we drop the subscript when $r=2$.

\subsection{The family of all double stars}

Gy\'arf\'as \cite{G77} proved that in every $r$-coloring of $K_n$ there is a monochromatic component of order at least $\frac{n}{r-1}$, which is best possible whenever an affine plane of order $r-1$ exists.  It turns out that a stronger statement is true; in every $r$-coloring of $K_n$ either there is a monochromatic component of order $n$ in every color, or there is a monochromatic double star of order at least $\frac{n}{r-1}$ (see \cite{LMP}, \cite{M}).  This led Gy\'arf\'as to raise the following problem. 

\begin{problem}[Gy\'arf\'as \cite{G11}] For all $r\geq 3$, is it true that in every $r$-coloring of the edges of $K_n$, there is a monochromatic double star of order at least $\frac{n}{r-1}$?  Equivalently, is it true that 
$$R_r(\cS_{2n+2})\leq (r-1)(2n+2)?$$
\end{problem}

Gy\'arf\'as and S\'ark\"ozy \cite{GS08} proved the following weaker bound for all $r\geq 2$,
\begin{equation}\label{eq:GS}
R_r(\cS_{2n+2})\leq \left(r-1+\frac{1}{r+1}\right)(2n+2)-\frac{r-1}{r+1}.
\end{equation}

This bound was later improved by S\'ark\"ozy \cite{S22} who proved that for all $r\geq 3$, there exists $0<\ep=O(\frac{1}{r^9})$ such that $$R_r(\cS_{2n+2})\leq \left(r-1+\frac{1}{r+1}-\ep\right)(2n+2).$$

\subsection{Balanced double stars}

A result of Grossman, Harary and Klawe \cite{GHK} implies $R(S_{n,n})=3n+2$ (see Section \ref{sec:conclusion} for further discussion about their general result).  For more than 2 colors, essentially nothing is known directly about $R_r(S_{n,n})$.  However, there are some existing extremal results which provide the following upper and lower bounds for all $r\geq 3$, 
\begin{equation}\label{eq:bounds}
(r-1)2n+1 \leq  R_r(S_{n,n})\leq r\cdot2n+2.
\end{equation}

First we describe the lower bound examples, both of which are well known.  The second example \cite{SYXL} is typically stated for paths, but we phrase it more generally here.

\begin{example}\label{ex:path} Let $r\geq 3$ be an integer.
\begin{enumerate}
\item If an affine plane of order $r-1$ exists and $r-1$ divides $2n$, then for every connected graph $G$ on $2n+2$ vertices, $R_r(G)\geq (r-1)2n+2$.
\item If $G$ is a balanced bipartite graph on $2n+2$ vertices, then $R_r(G)\geq (r-1)2n+1$.
\end{enumerate}
\end{example}

\begin{proof}
(i) Whenever an affine plane of order $r-1$ exists and $r-1$ divides $2n$, we blow up each of the $(r-1)^2$ points of the affine plane into a set of order $\frac{2n}{r-1}$ (allowing one of the sets to have order $\frac{2n}{r-1}+1$) to get an $r$-coloring where each color class consists of components of order $2n$ and one component of order $2n+1$.  

(ii) For all $r\geq 3$, the example is as follows.  Take $2r-2$ sets of order $n$, call them $X_1, X_2, \dots, X_{2r-2}$.  For all $i\in [r-1]$, color all edges inside $X_{i}\cup X_{r-1+i}$ with color $r$.  For all $i\in [r-1]$ color all edges from $X_{i}$ to $X_{i+1}\cup \dots \cup X_{i+r-2}$ and all edges from $X_{i+r-1}$ to $X_{i+r}\cup \dots \cup X_{i+2r-2}$ with color $i$.  Note that this decomposes the edges of $K_n$ into cliques of order $2n$ and complete bipartite graphs $K_{n,(r-2)n}$.  Thus we have no monochromatic copy of any balanced bipartite graph on $2n+2$ vertices.  
\end{proof}

The upper bound in \eqref{eq:bounds} follows from a known special case of the Erd\H{o}s-S\'os conjecture.

\begin{observation}\label{obs:ES}
For all $r\geq 1$, if $G$ is a graph on $r\cdot2n+2$ vertices with $e(G)\geq \frac{1}{r}\binom{r\cdot2n+2}{2}$, then $S_{n,n}\subseteq G$.
\end{observation}

We give the proof for expository purposes and we begin by stating the following well-known folklore lemma (which is proved by deleting vertices of degree at most $d/2$ until we are left with the desired subgraph).

\begin{lemma}\label{lem:avdeg}
If $G$ is a graph with average degree at least $d>0$, then $G$ has a subgraph $G'$ with average degree at least $d$ and minimum degree greater than $d/2$.  
\end{lemma}


\begin{proof}[Proof of Observation \ref{obs:ES}]
We have that the average degree of $G$ is at least $\frac{2\frac{1}{r}\binom{r\cdot2n+2}{2}}{r\cdot2n+2}=2n+\frac{1}{r}$.  So by Lemma \ref{lem:avdeg}, $G$ has a subgraph $G'$ with average degree greater than $2n$ and minimum degree greater than $n$.  In $G'$, let $u$ be a vertex of degree at least $2n+1$ and let $v$ be any neighbor of $u$.  Since $d(v)\geq n+1$ we get a copy of $S_{n,n}$ with centers $u$ and $v$.
\end{proof}

Our first main result is an improvement on the upper bound in \eqref{eq:bounds}.

\begin{theorem}\label{thm:main1}
For all $r\geq 2$ and $n\geq 1$, $R_r(S_{n,n})\leq (r-\frac{1}{2})(2n+2)-1$.
\end{theorem}

Note that when $r=2$, this matches the result from \cite{GHK}.  So at the moment, we have no guess as to whether $R_r(S_{n,n})$ is closer to $(r-\frac{1}{2})(2n+2)-1$ or $(r-1)2n+1$ for $r\geq 3$.

Our lower bound on $R_r(S_{n,n})$ comes from the more general lower bound on the $r$-color Ramsey number of balanced bipartite graphs.  One might be able to improve the lower bound on $R_r(S_{n,n})$ by taking advantage of the specific structure of double stars (c.f. Example \ref{ex:small}).

\begin{problem}
Is it true that for all $r\geq 3$ there exists $\ep>0$ such that $R_r(S_{n,n})>(r-1+\ep)2n$?  
\end{problem}

\subsection{Bipartite version}

Given a bipartite graph $G$ and a positive integer $r$, $R^{\bip}_r(G)$ is the smallest integer $N$ such that in every $r$-coloring of the edges of $K_{N,N}$ there is a monochromatic copy of $G$.

Recently, Decamillis and Song \cite{DS} proved the following extremal result for double stars in balanced bipartite graphs.

\begin{theorem}[Decamillis and Song \cite{DS}]
Let $G$ be a balanced bipartite graph on $2N$ vertices and let $n\geq m$ with $N\geq 3n+1$.  If $e(G)>\max\{nN, 2m(N-m)\}$, then $S_{n,m}\subseteq G$.  Furthermore, this result is best possible.
\end{theorem} 

From this, they obtained the following corollary.
\begin{corollary}[Decamillis and Song \cite{DS}]\label{cor:DS}Let $r\geq 2$ be an integer.
\begin{enumerate}
\item If $n\geq 2m$, then $R^{\bip}_r(S_{n,m})\leq rn+1$.
\item If $m\leq n<2m$, then $R^{\bip}_r(S_{n,m})\leq (r+\sqrt{r(r-2)})m+1= (2r-1-\frac{1}{2r}-O(\frac{1}{r^2}))m+1$.
\end{enumerate}
\end{corollary}

We note that Corollary \ref{cor:DS}(i) follows immediately from Lemma \ref{lem:avdeg} since the majority color class has average degree greater than $n$, so there is a subgraph in which there exists a vertex of degree at least $n+1$ whose neighbors on the other side all have degree greater than $n/2\geq m$ which implies they all have degree at least $m+1$.

Our second main result is an improvement on the upper bound of $R^{\bip}_r(S_{n,n})$.

\begin{theorem}\label{thm:bipmain}~
\begin{enumerate}
\item For all $r\geq 2, R^{\bip}_r(S_{n,n})\leq \left(\frac{3r-5 + \sqrt{r^2-2r+9}}{2}\right)n+1= (2r-3+\frac{2}{r}+O(\frac{1}{r^2}))n$

\item $R^{\bip}_3(S_{n,n})< 3.6678 n$
\end{enumerate}
\end{theorem}

Note that when $r=2$, we have $\left(\frac{3r-5 + \sqrt{r^2-2r+9}}{2}\right)n+1=2n+1$, so this recovers the known bound \cite{HJ} in that case.

Regarding lower bounds on $R^{\bip}_r(S_{n,n})$, first note that  $R^{\bip}_r(S_{n,n})\geq rn+1$ by taking a proper $r$-edge-coloring of $K_{r,r}$ and blowing up each vertex into a set of $n$ vertices.  In fact, this same example shows that $R^{\bip}_r(S_{n,m})\geq rn+1$ for $n\geq m$ (and thus Corollary \ref{cor:DS}(i) is tight).

A result of DeBiasio, Gy\'arf\'as, Krueger, Ruszink\'o, and S\'ark\"ozy \cite{DGKRS} implies a better lower bound on $R^{\bip}_r(S_{n,n})$ for all $r\geq 4$.

\begin{example}\label{ex:DGKRS}
For every balanced bipartite graph $G$ on $2n+2$ vertices,  
\[ R^{\bip}_r(G) \geq 
    \begin{cases}
        rn+1, & 1\leq r\leq 3\\
       5n+1, & r=4\\
       (2r-4)n+1, & r\geq 5 
        \end{cases}
  \]
\end{example}

We provide an alternate lower bound, specific to balanced double stars, which beats the lower bound from Example \ref{ex:DGKRS} when $r=3$ and $r=5$ (and matches the lower bound from Example \ref{ex:DGKRS} when $r=4$ and $r=6$).

\begin{example}\label{ex:small}For all $r\geq 2$,
\[
R^{\bip}_r(S_{n,n})\geq
    \begin{cases}
       (\frac{3r}{2}-1)n+1, & r \text{ is even}\\
       (r-1+\frac{\sqrt{r^2-1}}{2})n - \frac{r+1}{2}& r \text{ is odd} 
        \end{cases} 
\]
\end{example}

Thus by combining Theorem \ref{thm:bipmain} together with Example \ref{ex:DGKRS} and Example \ref{ex:small}, we have that for all $r\geq 3$,
\[
    \begin{rcases}
    (2+\sqrt{2})n-2 & r=3\\
       5n+1 & r=4\\
       (4+\sqrt{6})n-3 & r= 5\\
       (2r-4)n+1 & r\geq 6
        \end{rcases} 
 \leq 
 R^{\bip}_r(G) \leq
 \begin{cases}
    3.6678n & r=3\\
       \left(\frac{3r-5 + \sqrt{r^2-2r+9}}{2}\right)n+1 & r\geq 4
        \end{cases}.
  \]

\subsection{Comparison between balanced double stars and paths}

It is interesting to compare the Ramsey numbers of the double star $S_{n,n}$ to another well-studied balanced tree, the path on $2n+2$ vertices, $P_{2n+2}$. In the $r=2$ case, the Ramsey numbers of both graphs are known to be the same \cite{GG}, \cite{GHK}. Likewise, the lower bounds in the case of $r=3$ are the same.  Furthermore, for all $r\geq 4$, the best known bounds for $R_r(P_{2n+2})$ \cite{KS}  and $R_r(S_{n,n})$ (Theorem \ref{thm:bipmain}) are essentially the same.

For bipartite Ramsey numbers, the situation for small $r$ is quite different. For $r=2$, the bipartite Ramsey numbers for both graphs are the same \cite{FSc,GL}, \cite{HJ}. However, when $r=3$, it is known that $3n+1\leq R_3^{\bip}(P_{2n+2})=(3+o(1))n$ \cite{BLS1}, whereas we prove  that $(2+\sqrt{2})n+1\leq R_3^{\bip}(S_{n,n})\leq 3.6678n+1$.

One takeaway from this comparison is that, given the current state of knowledge, it is possible that for all $r\geq 2$, $R_r(S_{n,n})=R_r(P_{2n+2})$; however, it is impossible (due to the case $r=3$) that for all $r\geq 2$, $R_r^{\bip}(S_{n,n})=(1+o(1))R_r^{\bip}(P_{2n+2})$.

Another particular case of interest regarding the bipartite Ramsey numbers of $P_{2n+2}$ and $S_{n,n}$ is when $r=5$.  Note that Example \ref{ex:DGKRS} gives a lower bound of $6n+1$ for both graphs. However, Buci\'c, Letzter, and Sudakov improved this lower bound to $R_5^{\bip}(P_{2n+2})\geq 6.5n+1$.  Their example relies on the fact that $P_{2n+2}$ has vertex cover number equal to $n+1$ (whereas $S_{n,n}$ has vertex cover number 2).  In Example \ref{ex:small}, we improve the lower bound to $R_5^{\bip}(S_{n,n})\geq 6.4494n$.  Our example relies on the fact that $S_{n,n}$ has adjacent vertices each of degree $n+1$ (whereas $P_{2n+2}$ has maximum degree 2). 

\begin{table}[h]
\caption{A summary of results regarding the Ramsey numbers of paths and balanced double stars.  The new results from this paper correspond to the shaded entries.}
\begin{tabular}{|c|c|c|c|c|}
\hline
  & \multicolumn{2}{|c|}{$R_r(P_{2n+2})$} & \multicolumn{2}{|c|}{$R_r(S_{n,n})$} \\ \hline
$r$ & lower bound  & upper bound  & lower bound  & upper bound  \\ \hline
2 & $3n+2$ \cite{GG}& $3n+2$ \cite{GG}& $3n+2$ \cite{GHK}& $3n+2$ \cite{GHK}\\ \hline
3 & $4n+2$  \cite{GRSS}&  $4n+2$ \cite{GRSS}&  $4n+2$ \cite{GRSS} & \cellcolor{violet!20} $5n+1$   \\ \hline
$\geq 4$ &  $(r-1)2n+1$  \cite{SYXL} & $(r-\frac{1}{2}+o(1))2n$  \cite{KS}& $(r-1)2n+1$ \cite{SYXL}& $ \cellcolor{violet!20} (r-\frac{1}{2})(2n+2)-1$ \\ \hline
\end{tabular}
\bigskip

\begin{tabular}{|c|c|c|c|c|}
\hline
  & \multicolumn{2}{|c|}{$R_r^{\bip}(P_{2n+2})$} & \multicolumn{2}{|c|}{$R_r^{\bip}(S_{n,n})$} \\ \hline
$r$ & lower bound  & upper bound  & lower bound  & upper bound  \\ \hline
2 & $2n+1$ \cite{FSc,GL}&  $2n+1$ \cite{FSc,GL}& $2n+1$ \cite{HJ}& $2n+1$ \cite{HJ}\\ \hline
3 & $3n+1$ \cite{BLS1}& $(3+o(1))n \cite{BLS1}$ & \cellcolor{violet!20} $3.4142n$  & \cellcolor{violet!20}$3.6678n$   \\ \hline
4 & $5n+1$ \cite{DGKRS}& $(5+o(1))n$ \cite{BLS2}& $5n+1$ \cite{DGKRS}& \cellcolor{violet!20} $5.5616n$\\ \hline
5 & $6.5n+1$ \cite{BLS2}& $(7+o(1))n$ \cite{BLS2}& \cellcolor{violet!20}  $6.4494n$ & \cellcolor{violet!20} $7.4495n$ \\ \hline
$\geq 6$ & $(2r-4)n+1 $ \cite{DGKRS}& $(2r-3.5+O(\frac{1}{r}))n$ \cite{BLS2}&  $(2r-4)n+1 $ \cite{DGKRS}& \cellcolor{violet!20} $(2r-3+O(\frac{1}{r}))n$ \\ \hline
\end{tabular}
\end{table}

\section{Multicolor Ramsey numbers of balanced double stars}
In this section we prove $R_r(S_{n,n})\leq (r-\frac{1}{2})(2n+2)-1$.  We begin our proof by showing that if we have an $r$-coloring of $K_{(r-\frac{1}{2})2n+2}$ with no monochromatic $S_{n,n}$, then every vertex has degree at least $n+1$ in every color.  However, we do this indirectly by showing that if there is no monochromatic $S_{n,n}$, then every vertex has degree at most $2n$ in every color which in turn implies that every vertex has degree at least $n+1$ in every color\footnote{Note that if $N\leq (r-\frac{1}{2})2n+1$, we don't necessarily have this property anymore (since $2n(r-1)+n=(r-\frac{1}{2})2n$) and thus proving that all vertices have degree at most $2n$ in all colors doesn't automatically imply that all vertices have degree at least $n+1$ in every color.}.

Given a graph $G$ on $N$ vertices, let $L(G)=\{v\in V(G): d(v)\geq 2n+1\}$, $M(G)=\{v\in V(G): n+1\leq d(v)\leq 2n\}$, and $S(G)=\{v\in V(G): d(v)\leq n\}$ (we think of $L(G)$ as the set of vertices of large degree, $M(G)$ as the set of vertices of medium degree, and $S(G)$ as the set of vertices of small degree).  

\begin{observation}\label{obs:SL}
Let $G$ be a graph on $N\geq 2n+2$ vertices.  If $L(G)\neq \emptyset$, then either $G$ contains every double star on $2n+2$ vertices, or $\sum_{v\in L(G)}d(v)=e(L(G),S(G))$ (in particular, $S(G)\neq \emptyset$).  
\end{observation}

\begin{proof}
If there is an edge with one endpoint in $L(G)$ and the other in $M(G)\cup L(G)$, then we would have every double star $S_{n_1,n_2}$ on $2n+2$ vertices.  So if $L(G)\neq \emptyset$, then we must have $\sum_{v\in L(G)}d(v)=e(L(G),S(G))$; in particular, $S(G)\neq \emptyset$.
\end{proof}

\begin{lemma}\label{emptyLall}
Let $r\geq 2$ and $N\geq (r-\frac{1}{2})2n+2$. In every $r$-coloring of $K_N$ either
\begin{enumerate}
\item there is a monochromatic copy of every double star on $2n+2$ vertices, or
\item every vertex has degree at least $n+1$ and at most $2n$ in every color.
\end{enumerate}
\end{lemma}

\begin{proof}
Suppose we have an $r$-colored $K_N$ and for all $i\in[r]$, let $G_i$ be the graph on $V(K_N)$ consisting of edges of color $i$.  For all $i\in [r]$, set $L_i=L(G_i)$, $M_i=M(G_i)$, and $S_i=S(G_i)$.  Set $S=\cup_{i\in [r]}S_i$ and $L=\cup_{i\in [r]}L_i$.  Note that if $v$ has degree at most $n$ in some color $i\in [r]$, then $v$ has degree at least $\frac{N-1-n}{r-1}>2n$ in some other color $j\in [r]\setminus \{i\}$ and thus
\begin{equation}\label{eq:SL}
S\subseteq L.
\end{equation}

Now suppose that (i) fails.  Given the definitions of $S$ and $L$, we have that (ii) is equivalent to saying $S=\emptyset=L$, which by \eqref{eq:SL} is equivalent to saying $L=\emptyset$.  So suppose for contradiction that $L\neq \emptyset$ and consequently by Observation \ref{obs:SL}, $S\neq \emptyset$.  By Observation \ref{obs:SL} and the definition of $S_i$ we have 
\begin{equation}\label{eq:LiSi}
\sum_{v\in L_i}d_i(v)=e_i(S_i,L_i)\leq |S_i|n.
\end{equation}

For all $v\in V(G)$, let $\lambda_v=\{i\in [r]:v\in L_i\}$ and $\sigma_v=\{i\in [r]:v\in S_i\}$.  Note that for all $v\in V(G)$, we have $\sum_{i\in \lambda_v}d_i(v)\geq (2n+1)\lambda_v$, but the following alternate bound will prove more useful 
\begin{equation}\label{eq:vlambda2}
\sum_{i\in \lambda_v}d_i(v)\geq (r-\frac{1}{2})2n+1-n\sigma_v-2n(r-\sigma_v-\lambda_v)=(2\lambda_v+\sigma_v-1)n+1.
\end{equation}
For all $v\in L$, we have $\lambda_v\geq 1$ and thus $(2\lambda_v+\sigma_v-1)n+1>\sigma_vn$ which implies
\begin{equation}\label{eq:double}
\sum_{v\in L}(2\lambda_v+\sigma_v-1)n+1>\sum_{v\in L}\sigma_vn \stackrel{\eqref{eq:SL}}{=} \sum_{i\in [r]}|S_i|n.
\end{equation}

Putting everything together we have
\begin{align*}
\sum_{i\in [r]}|S_i|n\stackrel{\eqref{eq:LiSi}}{\geq} \sum_{i\in [r]}\sum_{v\in L_i}d_i(v)=\sum_{v\in L}\sum_{i\in \lambda_v}d_i(v)&\stackrel{\eqref{eq:vlambda2}}{\geq} \sum_{v\in L}(2\lambda_v+\sigma_v-1)n+1\stackrel{\eqref{eq:double}}{>}\sum_{i\in [r]}|S_i|n,
\end{align*}
a contradiction.
\end{proof}

Finally we prove the main result of this section. 

\begin{proof}[Proof of Theorem \ref{thm:main1}]
Let $N= (r-\frac{1}{2})(2n+2)-1$ and consider an arbitrary $r$-coloring of $K_N$.  If we don't have a monochromatic copy of every double star on $2n+2$ vertices, then since $N\geq (r-\frac{1}{2})2n+2$,  Lemma \ref{emptyLall} implies that every vertex has degree between $n+1$ and $2n$ in every color.  Since $N=(r-\frac{1}{2})(2n+2)-1\geq \left(r-1+\frac{1}{r+1}\right)(2n+2)-\frac{r-1}{r+1}$ (where the inequality holds for all $r\geq 2$ and $n\geq 1$), we have by \eqref{eq:GS} a monochromatic double star $S:=S_{n_1,n_2}$ of color $i$ and order at least $2n+2$ with $n_1\geq n_2$.  Suppose $xy$ is the central edge of $S$, and without loss of generality suppose $x$ is adjacent to $n_1$ many leaves and $y$ is adjacent to $n_2$ many leaves.  Since $y$ has degree at least $n+1$ in color $i$, and since $n_1-(n-n_2)\geq n$, there exists a copy of $S_{n,n}$ of color $i$ having $xy$ as the central edge.  
\end{proof}

Because of the slack between $(r-\frac{1}{2})(2n+2)-1$ and $\left(r-1+\frac{1}{r+1}\right)(2n+2)-\frac{r-1}{r+1}$ in the above proof, any improvement to the multiplicative term $(r-\frac{1}{2})$ in Lemma \ref{emptyLall}, will translate to an improvement in Theorem \ref{thm:main1}.  Since a slightly weaker statement will suffice, we state the problem more formally.

\begin{problem}
Is is true that for all $r\geq 3$, there exists $\ep>0$ such that in every $r$-coloring of the edges of $K_{(r-\frac{1}{2}-\ep)(2n+2)}$ either there is a monochromatic copy of $S_{n,n}$, or every vertex has degree at least $n+1$ in every color?
\end{problem}

\section{Bipartite case}

We begin with Example \ref{ex:small}. The following lemma describes a coloring which will be useful in our construction.

\begin{lemma}\label{lem:2col}
    Let $t\ge s > n\ge 2$ be integers and $G$ be a complete bipartite graph with parts $X$ and $Y$ of sizes $t$ and $s$ respectively. If $s - \floor{\frac{sn}{t}}\le n$, then there is a coloring of the edges of $G$ with colors $\set{1,2}$ such that 
    \begin{enumerate}
        \item $d_1(v) \le n$ for all $v\in X$, and 
        \item $d_2(v) \le n$ for all $v \in Y $.
    \end{enumerate}
\end{lemma}
\begin{proof}
    Let $X = \set{x_1,\ldots, x_t}$ and $Y = \set{y_1,\ldots, y_s}$. For each $i\in [s]$, let $y_i$ have edges of color 2 to vertices $x_{(i-1)n +1}, x_{(i-1)n +2}\ldots, x_{in}$ where the indices are taken modulo $t$. Color the remaining edges of $G$ with color 1. Condition (ii) is then satisfied by construction. To show condition (i), note that for all $v\in X$, we have $d_2(v)$ is either $\floor{\frac{sn}{t}}$ or $\ceiling{\frac{sn}{t}}$. So for all $v\in X$, $d_1(v) = s - d_2(v) \le s - \floor{\frac{sn}{t}}$ which is at most $n$ by assumption.
\end{proof}
Notice that when $t=s=2n$, this coloring gives two disjoint copies of $K_{n,n}$ in each color.

\begin{proof}[Proof of Example \ref{ex:small}] 
First suppose that $r$ is even.  Set $r=2k$ and set $N = (3k-1)n$. Partition the set of colors into two sets $A=\{1,\ldots,k\}$ and $B = \{k+1,\ldots 2k\}$. Also, let $A' = A\setminus\set{k}$ and $B' = B\setminus\set{2k}$  
We partition $X$ into $k$ sets $\{X_i:i\in A\}$, each of order $2n$ and $k-1$ sets $\{X_{j,2k}:j\in B'\}$, each of order $n$. Call a set ``single colored'' if it has one subscript and ``double colored'' if it has two. We similarly partition $Y$ into $k$ single colored sets $\{Y_j:j\in B\}$, each of order $2n$ and $k-1$ double colored sets $\{Y_{i,k}:i\in A'\}$, each of order $n$. The intention is that a vertex in a set $X_i$ (or $Y_i$) should have degree larger than $n$ in color $i$ and degree at most $n$ in all other colors. Likewise, vertices in $X_{i,j}$ (or $Y_{i,j}$) should have degree larger than $n$ in colors $i$ and $j$ and degree at most $n$ in all other colors.

Between $X_i$ and $Y_j$ we color as described in Lemma \ref{lem:2col} so that $d_{j}(v) \le n$ for all $v\in X_i$ and $d_i(v)\le n$ for all $v\in Y_j$. The hypothesis of the lemma is easy to check as both sets have order $2n$.  Color all the edges between $X_{j,2k}$ and $Y_i$ with color $j$ unless $j=i$ in which case we use color $2k$. Color all edges between $Y_{j,k}$ and $X_i$ with color $j$ unless $j=i$ in which case we use color $k$. Finally, color all edges between $Y_{i,k}$ and $X_{j, 2k}$ with color $i$. (See Figure \ref{fig:coloring})

\begin{figure}[ht]
    \includegraphics[width=\textwidth]{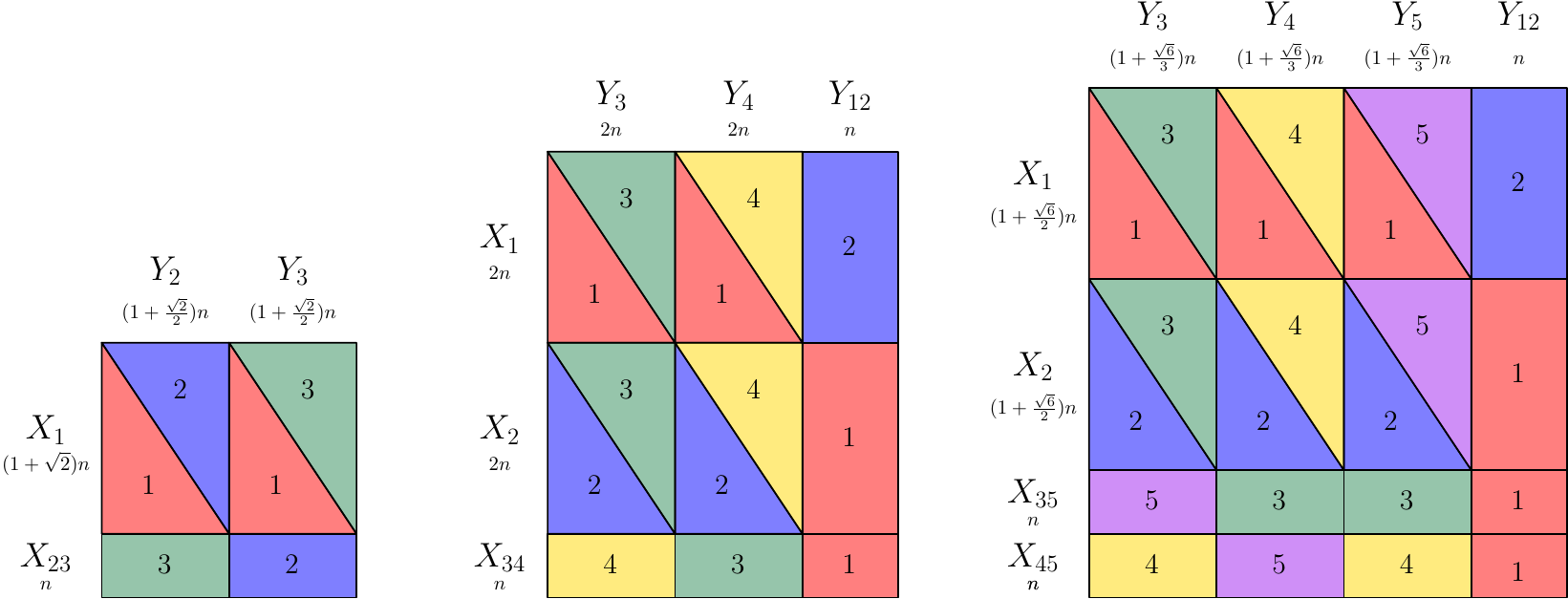}
    \caption{The edge coloring between sets used in the construction of the proof of Example \ref{ex:small} with three colors (left), four colors (center), and five colors (right). The size of each set is listed below the set name (ignoring floors and ceilings); for example, when $r=3$, $|X_1|= (1+\sqrt{2})n$.} 
    \label{fig:coloring}
\end{figure}

In this coloring, the monochromatic components incident to double colored vertices are all complete bipartite graphs with one side of order $n$. The monochromatic components between the $X_i$'s and $Y_j$'s are all of the type described in Lemma \ref{lem:2col}. Thus no monochromatic component contains a copy of $S_{n,n}$.

Now we consider the case when $r$ is odd. Set $r=2k-1$ and set $N =\floor{\alpha n}-k$ where $\alpha = r-1 + \frac{\sqrt{r^2 -1}}{2}$. We partition the set of colors into two sets $A = \set{1,\ldots, k-1}$ and $B = \set{k,\ldots, 2k-1}$. Let $A' = A\setminus\{k-1\}$ and $B' = B\setminus\{2k-1\}$. 
Now we partition $X$ into $k-1$ single colored sets $\set{X_i:i\in A}$, each of order $\ceiling{\frac{\alpha-(k-1)}{k-1}n}$ and $k-1$ double colored sets $\set{X_{j,2k-1}: j\in B'}$, each of order at most $n$. We partition $Y$ into $k$ single colored sets $\set{Y_j:j\in B}$, each of order $\floor{\frac{\alpha - (k-2)}{k}n}$ and $k-2$ double colored sets $\set{Y_{i, k-1}:i\in A'}$, each of order at most $n$.

Between $X_i$ and $Y_j$ we color as described in Lemma \ref{lem:2col} so that $d_{j}(v) \le n$ for all $v\in X_i$ and $d_i(v)\le n$ for all $v\in Y_j$. To check that the hypothesis of Lemma \ref{lem:2col} is satisfied in this case, first note that $|X_i| = \ceiling{ \frac{\alpha-(k-1)}{k-1}n} =\ceiling{ \frac{k-1 + \sqrt{k(k-1)}}{k-1}n} =\ceiling{ \of{1 + \sqrt{\frac{k}{k-1}}}n }$ and similarly, $|Y_j| = \floor{\of{1 + \sqrt{\frac{k-1}{k}}}n }$.
Thus
\begin{align*}
    |Y_j| - \bigfloor{\frac{|Y_j|n}{|X_i|}} 
    \le \bigfloor{ \of{1 + \sqrt{\frac{k-1}{k}}}n} - \bigfloor{ \frac{ \of{1 + \sqrt{\frac{k-1}{k}}}}{ \of{1 + \sqrt{\frac{k}{k-1}}}}n }
    = n.
\end{align*}
where the last equality holds since $\frac{ \of{1 + \sqrt{\frac{k-1}{k}}}}{ \of{1 + \sqrt{\frac{k}{k-1}}}}=\sqrt{\frac{k-1}{k}}$.

Now color all the edges between $X_{j,2k-1}$ and $Y_i$ with color $j$ unless $j=i$ in which case we use color $2k-1$. Color all edges between $Y_{j,k-1}$ and $X_i$ with color $j$ unless $j=i$ in which case we use color $k-1$. Finally, color all edges between $Y_{i,k-1}$ and $X_{j, 2k-1}$ with color $i$. As in the even case, this coloring contains no monochromatic $S_{n,n}$.
\end{proof}

Now we prove the main result of this section.

\begin{proof}[Proof of Theorem \ref{thm:bipmain}]
Let $N$ be an integer with $N\geq rn+1$ and set $\alpha=\frac{N}{n}$.  Suppose $K_{N,N}$ has been $r$-colored with no monochromatic $S_{n,n}$.  Later we will assume that $N$ is larger, but first we deduce some properties that hold when we simply have $N\geq rn+1$.

We color a vertex $v$ with color $i$ if $v$ is adjacent to at least $n+1$ edges of color $i$ (note that a vertex may receive more than one color and since $N\geq rn+1$, every vertex receives at least one color).  
For all $i\in [r]$, let $z_i$ be the number of vertices which receive exactly $i$ many colors.  For all $\emptyset\neq S\subseteq [r]$, let $X_S$ and $Y_S$ be the set of vertices in $X$ and $Y$ respectively which are colored with exactly the colors in $S$ and let $x_S=|X_S|$ and $y_S=|Y_S|$.  For all $i\in [r]$, let $\mathcal{X}_i$ and $\mathcal{Y}_i$ be the set of vertices in $X$ and $Y$ respectively which receive color $i$ (and possibly other colors). For $A\subseteq X$, $B\subseteq Y$, and  $S\subseteq [r]$, let $e_S(A,B)$ be the number of edges between $A$ and $B$ which receive any color from $S$.

We call an edge \emph{important} if it has color $i$ and is incident with a vertex of color $i$. The crucial part of this definition is the following observation. An important edge of color $i$ is incident to exactly one vertex of color $i$, otherwise this edge would form the central edge of a monochromatic $S_{n,n}$. Let $e^*$ be the number of important edges. Define $\sigma$ such that $\sigma^2$ is the proportion of edges which are \emph{not} important.  We have 
\begin{equation}\label{eq:sigma_r}
\sigma^2N^2 = \sum_{\emptyset\neq S_1,S_2\subseteq [r]}e_{[r]\setminus(S_1\cup S_2)}(X_{S_1}, Y_{S_2}) \geq \sum_{i\in [r]}x_iy_i.
\end{equation}

Note that by the definition of $z_i$, we have that for all $i\in [r]$ and all vertices $v$ which receive exactly $i$ colors, $v$ is incident with at least $N-(r-i)n$ important edges.  Thus we have the following bounds on $e^*$,
\begin{equation}\label{eq:e*}
\sum_{i\in [r]}z_i(N-(r-i)n)\leq e^*= (1-\sigma^2)N^2.
\end{equation}

Our first claim gives an upper bound on the number of vertices which are colored with more than one color.  Note that a higher proportion of non-important edges causes a smaller proportion of the vertices to have more than one color.

\begin{claim}\label{clm:doubles}
$$\sum_{i=2}^rz_i\leq (2r-2-\alpha(1+\sigma^2))N$$
\end{claim}

\begin{proofclaim}
Expanding, canceling, and simplifying \eqref{eq:e*} gives
\[
\sum_{i=2}^rz_i\leq z_2+2z_3+\dots+(r-1)z_r\leq (2r-2-\alpha(1+\sigma^2))N.\qedhere
\]
\end{proofclaim}

The next claim gives an absolute upper bound on the order of an individual set $X_i$ or $Y_i$.

\begin{claim}\label{clm:XC_YC} For all $i\in [r]$ we have 
$x_i\leq \frac{N}{\alpha-(r-1)}$ and 
$y_i\leq \frac{N}{\alpha-(r-1)}$.
\end{claim}

\begin{proofclaim}
For all $i\in [r]$ we have 
$$x_i(N-(r-1)n)\leq e_i(X_i, Y) = e_i(X_i, Y-\mathcal{Y}_i)\leq n(N-|\mathcal{Y}_i|),$$
and thus 
$$
x_i\leq \frac{N -|\mathcal{Y}_i|}{\alpha-(r-1)} \le \frac{N}{\alpha-(r-1)} .  
$$
Likewise for $y_i$.
\end{proofclaim}

The final claim gives an upper bound on the number of vertices which receive exactly one color.  

\begin{claim}\label{clm:bigC}
Let $C\in \mathbb{R}^+$.  If there are exactly $t$ indices $i\in [r]$ such that $\max\{x_i,y_i\}\geq \frac{\sigma N}{C}$, then 
$$z_1=\sum_{i\in [r]}(x_i+y_i)\leq \left(\frac{t}{\alpha-(r-1)}+(r-t)\frac{\sigma}{C}+C\sigma\right)N.$$
\end{claim}

\begin{proofclaim}
First note that if $\sigma =0$, then $x_i > 0$ implies that $y_i = 0$ and vice versa. Hence Claim \ref{clm:XC_YC} implies that $z_1 \le \frac{r}{\alpha-(r-1)} N $ and so the claim holds in this case. So we may assume that $\sigma>0$ for the remainder.
Without loss of generality, suppose that $\max\{x_i,y_i\}\geq \frac{\sigma N}{C}$ for all $i\in [t]$ and $\max\{x_i,y_i\}< \frac{\sigma N}{C}$ for all $i\in [r]\setminus [t]$.  

Note that for all $i\in [t]$, we have $\max\{x_i,y_i\}\min\{x_i,y_i\}=x_iy_i$ and since $i\in [t]$, we have $\max\{x_i,y_i\}\geq \frac{\sigma N}{C}$ and thus
\begin{equation}\label{eq:min}
\min\{x_i,y_i\}\leq \frac{x_iy_i}{\frac{\sigma N}{C}}
\end{equation}

For all $i\in [r]\setminus [t]$, we have $\max\{x_i,y_i\}<\frac{\sigma N}{C}$ and thus $\frac{x_i}{\frac{\sigma N}{C}}, \frac{y_i}{\frac{\sigma N}{C}}<1$.  From this (and the fact that for all real numbers $0\leq a,b\leq 1$, we have $a+b\leq 1+ab$) we have 
\begin{equation}\label{eq:r-t}
x_i+y_i\leq \frac{\sigma N}{C}+\frac{x_iy_i}{\frac{\sigma N}{C}}.
\end{equation}

Using \eqref{eq:min} and \eqref{eq:r-t} together with Claim \ref{clm:XC_YC}, we have
\begin{align*}
z_1=\sum_{i\in [r]}(x_i+y_i)&=\sum_{i\in [t]}(\max\{x_i,y_i\}+\min\{x_i,y_i\})+\sum_{i\in [r]\setminus [t]}(x_i+y_i)\\
&\leq \frac{t}{\alpha-(r-1)}N+\sum_{i\in [t]}\frac{x_iy_i}{\frac{\sigma N}{C}}+\sum_{i\in [r]\setminus [t]}\frac{\sigma N}{C}+\frac{x_iy_i}{\frac{\sigma N}{C}}\\
&=\frac{t}{\alpha-(r-1)}N+(r-t)\frac{\sigma N}{C}+\sum_{i\in [r]}\frac{x_iy_i}{\frac{\sigma N}{C}}\\
&\stackrel{\eqref{eq:sigma_r}}{\leq} \frac{t}{\alpha-(r-1)}N+(r-t)\frac{\sigma N}{C}+C\sigma N,
\end{align*}
as desired.
\end{proofclaim}

Now we prove part (i) of Theorem \ref{thm:bipmain}.  Let $N$ be an integer with $N>\left(\frac{3r-5 + \sqrt{r^2-2r+9}}{2}\right)n$, set $\alpha=\frac{N}{n}$, and note that 
\begin{equation}\label{eq:alpha}
\alpha>\frac{3r-5 + \sqrt{r^2-2r+9}}{2}.
\end{equation}
We now combine Claim \ref{clm:doubles} and Claim \ref{clm:bigC} to get a contradiction with \eqref{eq:alpha}.

\noindent
\tbf{Case 1} ($\sigma =0$) 
Applying Claim \ref{clm:bigC} (with say $C=1$) we see that since $\sigma=0$ we have that there are exactly $r$ indices with $i\in [r]$ such that $\max\{x_i,y_i\}\geq 0=\frac{\sigma N}{C}$ and thus Claim \ref{clm:bigC} together with Claim \ref{clm:doubles} gives 
 \begin{align*}
2N=z_1+\sum_{i=2}^rz_i&\leq \frac{r}{\alpha-(r-1)}N+(2r-2-\alpha)N= \left(\frac{r}{\alpha-(r-1)}+2r-2-\alpha\right)N
 \end{align*}
 which contradicts \eqref{eq:alpha}.

\noindent
\tbf{Case 2} ($\sigma >0$) 

Set $C=(\alpha-(r-1))\sigma$.
Let $t$ be the number of indices where $\max\{x_i, y_i\}\geq \frac{\sigma N}{C}=\frac{N}{\alpha-(r-1)}$.  Now Claim \ref{clm:bigC} (with $C=(\alpha-(r-1))\sigma)$) and Claim \ref{clm:doubles} implies
\begin{align*}
2N=z_1+\sum_{i=2}^rz_i&\leq \left(\frac{t}{\alpha-(r-1)}+(r-t)\frac{\sigma }{C}+C\sigma\right)N +(2r-2-\alpha(1+\sigma^2))N\\
&= \left(\frac{t}{\alpha-(r-1)}+\frac{r-t}{\alpha-(r-1)} + (\alpha-(r-1))\sigma^2 + 2r-2-\alpha(1+\sigma^2)\right)N\\
&=\left(\frac{r}{\alpha-(r-1)}+2r-2-\alpha-\sigma^2(r-1)\right)N\\
&\leq \left(\frac{r}{\alpha-(r-1)}+2r-2-\alpha\right)N
\end{align*}
which, as before, contradicts \eqref{eq:alpha}.

Now we prove part (ii) of Theorem \ref{thm:bipmain}. Let $N$ be an integer with $N\geq 3.6678n$, set $\alpha=\frac{N}{n}$, and note that $\alpha\geq 3.6678$.  (The exact bound we will get from our calculations is actually the largest of the three real solutions to the cubic polynomial $4\alpha^3-20\alpha^2+19\alpha+2=0$.  However, the exact form of this solution is quite ugly, so we give the approximation $3.6678$ instead).  

First note that for any positive integer $k$, 
\begin{equation}\label{eq:maxsig}
\sigma(k-\alpha\sigma)\leq \frac{k^2}{4\alpha} 
\end{equation}
with the maximum occurring when $\sigma=\frac{k}{2\alpha}$.

When $\sigma=0$ we do the same as above, but note that since $r=3$, there is one side, say $X$, in which at most one of $\{X_1, X_2, X_3\}$ is non-empty.  This fact together with Claim \ref{clm:XC_YC} and Claim \ref{clm:doubles} implies  
\begin{align*}
N=|X|\leq \frac{1}{\alpha-2}N+(4-\alpha)N, 
\end{align*}
which is a contradiction when $\alpha>\frac{5+\sqrt{5}}{2}\approx 3.618$.

When $\sigma>0$, Claim \ref{clm:bigC} (with $C=1$) and Claim \ref{clm:doubles} imply
\begin{equation}\label{eq:r=3}
2N=z_1+(z_2+z_3)\leq (\frac{t}{\alpha-2}+(3-t)\sigma +\sigma )N+(4-\alpha(1+\sigma^2))N.
\end{equation}

If $t=0$, then \eqref{eq:r=3} simplifies to
\begin{align*}
2N=z_1+(z_2+z_3)\leq 4\sigma N+(4-\alpha(1+\sigma^2))N&=(4-\alpha+\sigma(4-\alpha\sigma))N\\
&\stackrel{\eqref{eq:maxsig}}{\leq} (4-\alpha+\frac{4}{\alpha})N,
\end{align*}
which is a contradiction when $\alpha>1+\sqrt{5}\approx 3.2361$.

When $t\geq 1$, note that there is some set $W\in \{X_1, X_2, X_3, Y_1, Y_2, Y_3\}$ which has order at least $\sigma N$.  So by Claim \ref{clm:XC_YC}, we have $\sigma N\leq  |W|\leq \frac{1}{\alpha-2}N,$ and thus
 \begin{equation}\label{eq:sig_upper}
 \sigma\leq \frac{1}{\alpha-2}.
 \end{equation}

Now when $1\leq t\leq 2$, \eqref{eq:r=3} gives us
\begin{align*}
2N=z_1+(z_2+z_3)&\leq (\frac{t}{\alpha-2}+(3-t)\sigma +\sigma )N+(4-\alpha(1+\sigma^2))N\\
&\stackrel{\eqref{eq:sig_upper}}{\leq} \left(\frac{2}{\alpha-2}+4-\alpha+\sigma(2-\alpha\sigma)\right)N\\
&\stackrel{\eqref{eq:maxsig}}{\leq}\left(\frac{2}{\alpha-2}+4-\alpha+\frac{1}{\alpha}\right)N
\end{align*}
which is a contradiction when $\alpha>\frac{3+\sqrt{17}}{2}\approx 3.5616$.

Finally when $t=3$, we may suppose without loss of generality that $x_1$, $y_2$, and $y_3$ are at least $\sigma N$. 
Thus \[\sigma^2 N^2 \ge \sum_{i=1}^3 x_iy_i = \sum_{i=1}^3 \min\{x_i,y_i\}\max\{x_i,y_i\} \ge  \sigma N\sum_{i=1}^3 \min\{x_i,y_i\}, \]
which implies
\begin{equation}\label{eq:x2x3}
x_2+x_3\leq y_1+x_2+x_3\leq \sigma N. 
\end{equation}

Now by Claim \ref{clm:XC_YC} and \eqref{eq:x2x3} we have
\begin{align*}
N=|X|=x_1+(x_2+x_3)&\leq \frac{N}{\alpha-2}+\sigma N+(4-\alpha(1+\sigma^2))N\\
&= (\frac{1}{\alpha-2}+4-\alpha+\sigma(1-\alpha\sigma))N\\
&\stackrel{\eqref{eq:maxsig}}{\leq} (\frac{1}{\alpha-2}+4-\alpha+\frac{1}{4\alpha})N
\end{align*}
which is a contradiction when $\alpha\geq 3.6678$.
\end{proof}

\section{Conclusion and open problems}\label{sec:conclusion}

Aside from improving the bounds for balanced double stars, one of the directions of further study would be to consider the case of unbalanced stars.

Grossman, Harary and Klawe \cite{GHK} proved that
\[ R(S_{n, m}) = 
    \begin{cases}
        \max\{2n+1, n+2m+2\}, & \text{if $n$ is odd and $m\leq 2$}\\
        \max\{2n+2, n+2m+2\}, & \text{if $n$ is even or $m\geq 3$, and $n\leq \sqrt{2}m$ or $n\geq 3m$ }
        \end{cases}
  \]
and they conjectured that their result should also hold in the range when $\sqrt{2}m<n<3m$.  However, Norin, Sun, and Zhao \cite{NSZ} disproved this conjecture -- a particular case of interest is when $n=2m$ and in this case, they showed that $S_{2m, m}\geq 4.2m$.  Very recently, Flores Dub\'o and Stein \cite{FS} proved that $S_{2m, m}\leq 4.275m$.  See \cite{NSZ,FS} for a discussion about the best known bounds in general when $\sqrt{2}m<n<3m$.

In the multicolor case, Ruotolo and Song \cite{RS} proved the following.
\begin{theorem}
Let $n\geq m\geq 1$ and $r\geq 1$ be integers with $m= O(\frac{n}{r^2})$.  
\begin{enumerate}
\item If $r$ is odd, then $R_r(S_{n,m})=rn+m+2$.  
\item If $r$ is even, then $\max\{rn+1,~ (r-1)n+2m+2\}\leq R_r(S_{n,m})\leq rn+m+2$.  
\end{enumerate}
\end{theorem}
It would be interesting to see if their result extends to the case when $m=O(\frac{n}{r})$.  

We also note that our proof of Theorem \ref{thm:main1} actually gives a monochromatic copy of all double stars $S_{n_1,n_2}$ where $n_1+n_2=2n$ and (informally) $|n_1-n_2|$ is small enough.  

Regarding lower bounds for unbalanced double stars, we collect all of the examples from \cite{RS} and Example \ref{ex:path} below.
\begin{example}
Let $r\geq 3$ and $n\geq m\geq 1$ be integers.
\begin{enumerate}
\item If $r$ is odd and an affine plane of order $r-1$ exists, $$R_r(S_{n,m})\geq \max\{rn+m+2,~ 2(r-1)m+1,~ (r-1)(n+m)+1\}.$$
\item If $r$ is even and an affine plane of order $r-1$ exists, $$R_r(S_{n,m})\geq \max\{rn+1,~ (r-1)n+2m+2,~ 2(r-1)m+1,~ (r-1)(n+m)+1\}.$$
\end{enumerate}
\end{example}

\begin{proof}~
\begin{itemize}
\item $rn+m+2$ comes from \cite[Theorem 2.1(a)]{RS}
\item $(r-1)n+2m+2$ comes from \cite[Theorem 2.1(b)]{RS}
\item $rn+1$ comes from the lower bound on a star with $n$ leaves
\item $2(r-1)m+1$ comes from Example \ref{ex:path}(ii)
\item $(r-1)(n+m)+1$ comes from Example \ref{ex:path}(i)
\end{itemize}
\end{proof}

As mentioned in Corollary \ref{cor:DS}(i), it is known that if $2m\leq n$, then $R_r^{\bip}(S_{n,m})=rn+1$.  We studied the case when $n=m$. So it would be interesting to more carefully study the behavior of $R_r^{\bip}(S_{n,m})$ in the range $n< m<2n$.

\bibliographystyle{abbrv}
\bibliography{references}

\end{document}